\documentclass[12pt,twoside]{amsart}
\usepackage{amsmath, amsthm, amscd, amsfonts, amssymb, graphicx}
\usepackage{enumerate}
\usepackage[colorlinks=true,
linkcolor=blue,
urlcolor=cyan,
citecolor=red]{hyperref}
\usepackage{mathrsfs}
\newtheorem{theorem}{Theorem}[section]
\newtheorem{lemma}{Lemma}[section]
\newtheorem{remark}{Remark}[section]

\newtheorem{corollary}{Corollary}[section]

\numberwithin{equation}{section}

\begin{document}
	
\title{Singular values inequalities for matrix means}
\author{Mohammed Sababheh, Shigeru Furuichi, Shiva Sheybani and Hamid Reza Moradi}
\subjclass[2010]{Primary 47A63, Secondary 47A64, 47B15, 15A45.}
\keywords{positive matrices, matrix means, singular values.}

\begin{abstract}
In this article, we show multiple inequalities for the singular values of the difference of matrix means. The obtained results refine and complement some well established results in the literature. Although we target singular values inequalities, we will show several matrix means inequalities, as well.
\end{abstract}
\maketitle
\pagestyle{myheadings}
\markboth{\centerline {M. Sababheh, S. Furuichi, S. Sheybani and H. R. Moradi}}
{\centerline {Singular values inequalities for matrix means}}
\bigskip
\bigskip
\section{Introduction}
Let $\mathcal{M}_n^+$ denote the cone of positive $n\times n$ complex matrices. That is, $A\in\mathcal{M}_n^+$ if and only if $\left<Ax,x\right> >0 $ for all non zero vectors $x\in\mathbb{C}^n.$\\
Inequalities governing positive matrices have been in the center of numerous researchers' attention. Among the most well studied inequalities for positive matrices are those inequalities controlling matrix means.\\
Recall that when $A,B\in\mathcal{M}_n^+$, the weighted arithmetic, geometric and harmonic means of $A,B$ are defined respectively by
$$A\nabla_v B=(1-v)A+v B, A\sharp_v B=A^{\frac{1}{2}}\left(A^{-\frac{1}{2}}BA^{-\frac{1}{2}}\right)^{v}A^{\frac{1}{2}}, A!_vB=((1-v)A^{-1}+vB^{-1})^{-1},$$ for $0\leq v\leq 1.$
When $v=\frac{1}{2},$ we use the notations $\nabla, \sharp$ and $!$ instead of $\nabla_{\frac{1}{2}}, \sharp_{\frac{1}{2}}$ and $!_{\frac{1}{2}}$, respectively.

The notion of matrix means is too wide, but those three means happen to appear most frequently. It is well known that when $A,B\in\mathcal{M}_n^+$, one has the basic inequality 
\begin{equation}\label{amgm_intro}
A!_vB\leq A\sharp_vB\leq A\nabla_vB, 0\leq v\leq 1.
\end{equation}
Also, it is well known that when $A,B\in\mathcal{M}_n^+$ are such that $A\leq B$ then 
\begin{equation}\label{order_intro}
A\leq A!_vB, A\sharp_v B,A\nabla_vB\leq B.
\end{equation}
Although \eqref{amgm_intro} presents some relations among those means, it is of great interest to find better and sharper bounds. Further, computations of $A\sharp_vB$ is not as easy as $A\nabla_vB$ or $A!_vB.$ This urges the search for some relations that could be easier than just stating \eqref{amgm_intro}. We refer the reader to \cite{6,5,4} for further investigation of \eqref{amgm_intro}.

In \cite{1,2}, some singular values inequalities were given to describe the difference between such means. For example, it was shown in \cite{1} that if $A,B\in\mathcal{M}_n^+$ are such that $B\leq A$, then
\begin{equation}\label{gumus_1}
{{s}_{j}}\left( A\nabla B-A\sharp B \right)\le \frac{1}{8}{{s}_{j}}\left( {{B}^{-\frac{1}{2}}}{{\left( A-B \right)}^{2}}{{B}^{-\frac{1}{2}}} \right), j=1,2,\cdots,n,
\end{equation}
and 
\begin{equation}\label{gumus_2}
{{s}_{j}}\left( A\nabla B-A\sharp B \right)\geq \frac{1}{8}{{s}_{j}}\left( {{A}^{-\frac{1}{2}}}{{\left( A-B \right)}^{2}}{{A}^{-\frac{1}{2}}} \right), j=1,2,\cdots,n,
\end{equation}
where $s_j(X)$ represents the $j^{{\text{th}}}$ singular value of the matrix $X$, when all singular values of $X$ are arranged in a decreasing order, counting multiplicities.\\
These inequalities simulate the scalar inequalities
\begin{equation}\label{scalar_intro}
\frac{1}{8}\frac{(a-b)^2}{a}\leq \frac{a+b}{2}-\sqrt{ab}\leq \frac{1}{8}\frac{(a-b)^2}{b},
\end{equation}
 valid for the positive numbers $a\geq b,$ \cite{3}.

Following the same theme, it has been shown in \cite[Corollary 1]{2} that if $A,B\in\mathcal{M}_ n^+$ are  such that $B\le A$, then
\begin{equation}\label{hirz_intro}
\frac{v\left( 1-v \right)}{2}{{s }_{j}}\left( {{A}^{-\frac{1}{2}}}{{\left( A-B \right)}^{2}}{{A}^{-\frac{1}{2}}} \right)\leq {{s }_{j}}\left( A{{\nabla }_{v}}B-A{{\sharp}_{v}}B \right)\le \frac{v\left( 1-v \right)}{2}{{s }_{j}}\left( {{B}^{-\frac{1}{2}}}{{\left( A-B \right)}^{2}}{{B}^{-\frac{1}{2}}} \right)
\end{equation}
for any $0\le v\le 1$. Notice that substituting $v=\frac{1}{2}$ in \eqref{hirz_intro} implies \eqref{gumus_1} and \eqref{gumus_2}.

The main goal in this article is to present sharper and related inequalities for the singular values of the difference $A{{\nabla }_{v}}B-A{{\sharp}_{v}}B$. Adding to this, we present singular values inequalities for the differences $A\sharp_v B-A!_v B$ and $A\nabla B-A!B.$ We will notice how these different differences have similar bounds.
\section{Main Results}
In this section, we present our results, in different sections based on the means we are dealing with.
\subsection{ Arithmetic-Geometric mean inequalities}
We begin with the following matrix version of \eqref{scalar_intro}, without imposing any conditions on the order between $A$ and $B.$
\begin{theorem}\label{thm1}
Let $A,B\in\mathcal{M}_n^+$. Then 
\begin{equation}\label{01}
\begin{aligned}
\frac{1}{8}\left( A-B \right){{\left( A\nabla B \right)}^{-1}}\left( A-B \right)&\leq   A\nabla B-A\sharp B  & \le \frac{1}{8}\left( A-B \right){{\left( A\sharp B \right)}^{-1}}\left( A-B \right).  
\end{aligned}
\end{equation}
\end{theorem}
\begin{proof}
For any $x\ge 0$, it can be easily seen that
	\[\frac{1+x}{2}-\sqrt{x}=\frac{1}{8}{{\left( 1-x \right)}^{2}}{{\left( \frac{\frac{1+x}{2}+\sqrt{x}}{2} \right)}^{-1}}.\]
 By applying functional calculus for the operator ${{A}^{-\frac{1}{2}}}B{{A}^{-\frac{1}{2}}},$ we infer that
	\[\begin{aligned}
  & \frac{I+{{A}^{-\frac{1}{2}}}B{{A}^{-\frac{1}{2}}}}{2}-{{\left( {{A}^{-\frac{1}{2}}}B{{A}^{-\frac{1}{2}}} \right)}^{\frac{1}{2}}} \\ 
 & =\frac{1}{8}\left( I-{{A}^{-\frac{1}{2}}}B{{A}^{-\frac{1}{2}}} \right){{\left( \frac{\frac{I+{{A}^{-\frac{1}{2}}}B{{A}^{-\frac{1}{2}}}}{2}+{{\left( {{A}^{-\frac{1}{2}}}B{{A}^{-\frac{1}{2}}} \right)}^{\frac{1}{2}}}}{2} \right)}^{-1}}\left( I-{{A}^{-\frac{1}{2}}}B{{A}^{-\frac{1}{2}}} \right). \\ 
\end{aligned}\]
Thus,
	\[\begin{aligned}
  & \frac{I+{{A}^{-\frac{1}{2}}}B{{A}^{-\frac{1}{2}}}}{2}-{{\left( {{A}^{-\frac{1}{2}}}B{{A}^{-\frac{1}{2}}} \right)}^{\frac{1}{2}}} \\ 
 & =\frac{1}{8}\left( I-{{A}^{-\frac{1}{2}}}B{{A}^{-\frac{1}{2}}} \right){{A}^{\frac{1}{2}}}{{A}^{-1}}{{A}^{\frac{1}{2}}}{{\left( \frac{\frac{I+{{A}^{-\frac{1}{2}}}B{{A}^{-\frac{1}{2}}}}{2}+{{\left( {{A}^{-\frac{1}{2}}}B{{A}^{-\frac{1}{2}}} \right)}^{\frac{1}{2}}}}{2} \right)}^{-1}}{{A}^{\frac{1}{2}}}{{A}^{-1}}{{A}^{\frac{1}{2}}}\left( I-{{A}^{-\frac{1}{2}}}B{{A}^{-\frac{1}{2}}} \right). \\ 
\end{aligned}\]
Multiplying both sides by ${{A}^{\frac{1}{2}}}$ implies,
\begin{equation}\label{5}
A\nabla B-A\sharp B=\frac{1}{8}\left( A-B \right){{\left( \frac{A\nabla B+A\sharp B}{2} \right)}^{-1}}\left( A-B \right).
\end{equation}
It follows from the matrix arithmetic--geometric mean inequality \eqref{amgm_intro} that
\begin{equation}\label{1}
\begin{aligned}
   A\nabla B-A\sharp B&=\frac{1}{8}\left( A-B \right){{\left( \frac{A\nabla B+A\sharp B}{2} \right)}^{-1}}\left( A-B \right) \\ 
 & \le \frac{1}{8}\left( A-B \right){{\left( A\sharp B \right)}^{-1}}\left( A-B \right).  
\end{aligned}
\end{equation}
This proves the second  inequality in \eqref{01}. To prove the first inequality in \eqref{01},  \eqref{5} and \eqref{amgm_intro} imply
\[\begin{aligned}
   A\nabla B-A\sharp B&=\frac{1}{8}\left( A-B \right){{\left( \frac{A\nabla B+A\sharp B}{2} \right)}^{-1}}\left( A-B \right) \\ 
 & \ge \frac{1}{8}\left( A-B \right){{\left( A\nabla B \right)}^{-1}}\left( A-B \right).  
\end{aligned}\]
This completes the proof.
\end{proof}
We emphasize the identity 
\begin{equation}\label{identity_1}
A\nabla B-A\sharp B=\frac{1}{8}\left( A-B \right){{\left( \frac{A\nabla B+A\sharp B}{2} \right)}^{-1}}\left( A-B \right),
\end{equation}
which we have just obtained in the proof of Theorem \ref{thm1}.
\begin{remark}
Although Theorem \ref{thm1} is stated for positive matrices of order $n\times n$, it is still valid for positive operators $A,B$ on an infinite dimensional separable Hilbert space. 
\end{remark}
As a consequence of Theorem \ref{thm1}, we have the following singular value inequality.
\begin{corollary}
Let $A,B\in\mathcal{M}_n^+$. Then
\begin{equation}\label{2}
{{s}_{j}}\left( A\nabla B-A\sharp B \right)\le \frac{1}{8}{{s}_{j}}\left( {{\left( A\sharp B \right)}^{-\frac{1}{2}}}{{\left( A-B \right)}^{2}}{{\left( A\sharp B \right)}^{-\frac{1}{2}}} \right)
\end{equation}
and
\begin{equation}\label{4}
{{s}_{j}}\left( A\nabla B-A\sharp B \right)\ge \frac{1}{8}{{s}_{j}}\left( {{\left( A\nabla B \right)}^{-\frac{1}{2}}}{{\left( A-B \right)}^{2}}{{\left( A\nabla B \right)}^{-\frac{1}{2}}} \right),
\end{equation}
for $j=1,2,\cdots,n$.
\end{corollary}
\begin{proof}

From the second inequality in \eqref{01} and Weyl's monotonicity principle, we infer that
\begin{equation}\label{n1}
{{s}_{j}}\left( A\nabla B-A\sharp B \right)\le \frac{1}{8}{{s}_{j}}\left( \left( A-B \right){{\left( A\sharp B \right)}^{-1}}\left( A-B \right) \right)
\end{equation}
for $j=1,2,\cdots $.\\
 Since ${{s}_{j}}\left( {{X}^{*}}X \right)={{s}_{j}}\left( X{{X}^{*}} \right)$ for $j=1,2,\cdots $, it can be seen that
\[{{s}_{j}}\left( \left( A-B \right){{\left( A\sharp B \right)}^{-1}}\left( A-B \right) \right)={{s}_{j}}\left( {{\left( A\sharp B \right)}^{-\frac{1}{2}}}{{\left( A-B \right)}^{2}}{{\left( A\sharp B \right)}^{-\frac{1}{2}}} \right).\] This together with \eqref{n1} imply the first desired inequality. 

To prove the second inequality, we proceed similarly noting the first inequality in \eqref{01} and the fact that
\[{{s}_{j}}\left( \left( A-B \right){{\left( A\nabla B \right)}^{-1}}\left( A-B \right) \right)={{s}_{j}}\left( {{\left( A\nabla B \right)}^{-\frac{1}{2}}}{{\left( A-B \right)}^{2}}{{\left( A\nabla B \right)}^{-\frac{1}{2}}} \right)\]
for $j=1,2,\cdots $.
\end{proof}
Now we have the following refinements of \eqref{gumus_1} and \eqref{gumus_2}.
\begin{corollary}\label{cor_AG}
Let $A,B\in\mathcal{M}_n^{+}$ be such that $B\le A$. Then
\begin{equation}\label{6}
\begin{aligned}
   {{s}_{j}}\left( A\nabla B-A\sharp B \right)&\le \frac{1}{8}{{s}_{j}}\left( {{\left( A\sharp B \right)}^{-\frac{1}{2}}}{{\left( A-B \right)}^{2}}{{\left( A\sharp B \right)}^{-\frac{1}{2}}} \right) \\ 
 & \le \frac{1}{8}{{s}_{j}}\left( {{B}^{-\frac{1}{2}}}{{\left( A-B \right)}^{2}}{{B}^{-\frac{1}{2}}} \right)  
\end{aligned}
\end{equation}
and
\begin{equation}\label{7}
\begin{aligned}
   {{s}_{j}}\left( A\nabla B-A\sharp B \right)&\ge \frac{1}{8}{{s}_{j}}\left( {{\left( A\nabla B \right)}^{-\frac{1}{2}}}{{\left( A-B \right)}^{2}}{{\left( A\nabla B \right)}^{-\frac{1}{2}}} \right) \\ 
 & \ge \frac{1}{8}{{s}_{j}}\left( {{A}^{-\frac{1}{2}}}{{\left( A-B \right)}^{2}}{{A}^{-\frac{1}{2}}} \right)  
\end{aligned}
\end{equation}
for $j=1,2,\cdots,n $.
\end{corollary}
\begin{proof}
Since $B\le A$, we have $B\le A\sharp B\le A$. So, ${{A}^{-1}}\le {{\left( A\sharp B \right)}^{-1}}\le {{B}^{-1}}$. Now, from the inequality \eqref{1}, we get
\begin{equation}\label{3}
\begin{aligned}
   A\nabla B-A\sharp B&\le \frac{1}{8}\left( A-B \right){{\left( A\sharp B \right)}^{-1}}\left( A-B \right) \\ 
 & \le \frac{1}{8}\left( A-B \right){{B}^{-1}}\left( A-B \right).  
\end{aligned}
\end{equation}
From the inequality \eqref{3} and Weyl's monotonicity principle we have
\[\begin{aligned}
   {{s}_{j}}\left( A\nabla B-A\sharp B \right)&\le \frac{1}{8}{{s}_{j}}\left( \left( A-B \right){{\left( A\sharp B \right)}^{-1}}\left( A-B \right) \right) \\ 
 & \le \frac{1}{8}{{s}_{j}}\left( \left( A-B \right){{B}^{-1}}\left( A-B \right) \right)  
\end{aligned}\]
Since ${{s}_{j}}\left( {{X}^{*}}X \right)={{s}_{j}}\left( X{{X}^{*}} \right)$ for $j=1,2,\ldots $, it can be seen that
\[{{s}_{j}}\left( \left( A-B \right){{B}^{-1}}\left( A-B \right) \right)={{s}_{j}}\left( {{B}^{-\frac{1}{2}}}{{\left( A-B \right)}^{2}}{{B}^{-\frac{1}{2}}} \right).\]
Therefore,
\[\begin{aligned}
   {{s}_{j}}\left( A\nabla B-A\sharp B \right)&\le \frac{1}{8}{{s}_{j}}\left( {{\left( A\sharp B \right)}^{-\frac{1}{2}}}{{\left( A-B \right)}^{2}}{{\left( A\sharp B \right)}^{-\frac{1}{2}}} \right) \\ 
 & \le \frac{1}{8}{{s}_{j}}\left( {{B}^{-\frac{1}{2}}}{{\left( A-B \right)}^{2}}{{B}^{-\frac{1}{2}}} \right),  
\end{aligned}\]
and this proves \eqref{6}. To prove \eqref{7}, \eqref{amgm_intro} implies
\[\begin{aligned}
   A\nabla B-A\sharp B&\ge \frac{1}{8}\left( A-B \right){{\left( A\nabla B \right)}^{-1}}\left( A-B \right) \\ 
 & \ge \frac{1}{8}\left( A-B \right){{A}^{-1}}\left( A-B \right),  
\end{aligned}\]
as required.
\end{proof}

As a byproduct of Theorem \ref{thm1}, we have the following improvement of the second inequality in \eqref{amgm_intro}.
\begin{corollary}
Let $A,B\in\mathcal{M}_n^+$ be such that $A-B$ is invertible. Then
\begin{align*}
A\sharp B\leq \frac{1}{8}(A-B)(A\nabla B-A\sharp B)^{-1}(A-B)\leq A\nabla B. 
\end{align*}
In particular, if $A-B$ is invertible, then so is $A\nabla B-A\sharp B.$
\end{corollary}
\begin{proof}
Direct manipulations of \eqref{01} imply the desired result.
\end{proof}

Now we move to the study of the difference $A\nabla_vB-A\sharp_vB,$ rather than $A\nabla B-A\sharp B.$ The obtained results complement those in \cite{2}. First a lemma.
\begin{lemma}\label{lemma_v}
Let $x\geq 1$ and $0\leq v\leq 1.$ Then
\[\frac{v(1-v)}{2x}(x-1)^2\leq 
\left( 1-v \right)+vx-{{x}^{v}}\le \frac{v\left( 1-v \right)}{2}{{\left(x-1 \right)}^{2}}\left(\frac{2x}{x+1}\right).\]
\end{lemma}
\begin{proof}
We prove the first inequality. We define 
$$
g(x):=2x^{v+1} +v(1-v)(x-1)^2-2(1-v)x-2vx^2, x\geq 1. 
$$
Simple calculations imply
$g'(x)=2(v+1)x^v-2vx-2v^2x+2v^2-2$, $g''(x)=2v(v+1)\left(x^{v-1}-1\right) \leq 0$.
Thus we have $g'(x)\leq g'(1) =0$ which implies $g(x)\leq g(1)=0$. 

To prove the second inequality, we set 
$$
f(x):=(1-v)\left(1+\frac{1}{x}\right)+v(x+1)-x^{v-1}(x+1)-v(1-v)(x-1)^2, x\geq 1.
$$
By simple calculations, we have
$f'(x)=x^{-2}\left\{ x^v+2v^2(x-1)x^2 -1 -v\left(x^{v+1}+x^v+2x^3-3x^2-1\right) \right\}$ and $f''(x) =(1-v)\left\{ 2x^{-3}(1-x^v)+v\left(x^{v-2}+x^{v-3}-2\right)\right\} \leq 0$ since we have
$x^{v-2} \leq 1$ and $x^{v-3} \leq 1$ for $x \geq 1$ and $0\leq v \leq 1$. Thus we have $f'(x) \leq f'(1)=0$ which implies $f(x) \leq f(1)=0$.
\end{proof}
Manipulating Lemma \ref{lemma_v} implies the following bounds for the difference $A\nabla_vB-A\sharp_vB.$
\begin{theorem}\label{theorem1.2}
Let $A,B\in\mathcal{M}_n^+$ with $A\leq B$ and let $0\leq v\leq 1.$ Then
\begin{equation}\label{8}
\frac{v(1-v)}{2}(B-A)B^{-1}(B-A)\leq
A{{\nabla }_{v}}B-A{{\sharp}_{v}}B\le \frac{v\left( 1-v \right)}{2}\left( B-A \right)A^{-1}(A!B)A^{-1}\left( B-A \right).
\end{equation}
\end{theorem}
\begin{proof}
In Lemma \ref{lemma_v}, let $x=A^{-\frac{1}{2}}BA^{-\frac{1}{2}} \geq I.$ Then
\[\begin{aligned}
  &\frac{v(1-v)}{2} \left(A^{-\frac{1}{2}}BA^{-\frac{1}{2}}-I\right)
  A^{\frac{1}{2}}B^{-1}A^{\frac{1}{2}}
  \left(A^{-\frac{1}{2}}BA^{-\frac{1}{2}}-I\right)
  \leq \left( 1-v \right)I+v{{A}^{-\frac{1}{2}}}B{{A}^{-\frac{1}{2}}}-{{\left( {{A}^{-\frac{1}{2}}}B{{A}^{-\frac{1}{2}}} \right)}^{v}} \\ 
 & \le\frac{v\left( 1-v \right)}{2}\left( I-{{A}^{-\frac{1}{2}}}B{{A}^{-\frac{1}{2}}} \right){{A}^{\frac{1}{2}}}{{A}^{-1}}{{A}^{\frac{1}{2}}}{{\left( \frac{I+{{\left( {{A}^{-\frac{1}{2}}}B{{A}^{-\frac{1}{2}}} \right)}^{-1}}}{2} \right)}^{-1}}{{A}^{\frac{1}{2}}}{{A}^{-1}}{{A}^{\frac{1}{2}}}\left( I-{{A}^{-\frac{1}{2}}}B{{A}^{-\frac{1}{2}}} \right).
\end{aligned}\]
Multiply both sides by ${{A}^{\frac{1}{2}}}$ implies the desired result.
\end{proof}

In the following Lemma, we present the complement of Lemma \ref{lemma_v}, so that we can show a complement of Theorem \ref{theorem1.2}.
\begin{lemma}\label{lemma_vi}
Let $0< x\leq 1$ and $0\leq v\leq 1.$ Then
$$
\frac{v\left( 1-v \right)}{2}{{\left(x-1 \right)}^{2}}\left(\frac{2x}{x+1}\right)
\leq \left( 1-v \right)+vx-{{x}^{v}}\leq
 \frac{v(1-v)}{2x}(x-1)^2.
$$
\end{lemma}
\begin{proof}
To prove the first inequality, we set the function on $0< x \leq 1$,
$$
f(x):=(1-v)\left(1+\frac{1}{x}\right)+v(x+1)-x^{v-1}(x+1)-v(1-v)(x-1)^2.
$$
By simple calculations, we have
$f'(x)=x^{-2}\left\{ x^v+2v^2(x-1)x^2 -1 -v\left(x^{v+1}+x^v+2x^3-3x^2-1\right) \right\}$ and $f''(x) =(1-v)\left\{ 2x^{-3}(1-x^v)+v\left(x^{v-2}+x^{v-3}-2\right)\right\} \geq 0$ since we have
$x^{v-2} \geq 1$ and $x^{v-3} \geq 1$ for $0<x \leq 1$ and $0\leq v \leq 1$. Thus we have $f'(x) \leq f'(1)=0$ which implies $f(x) \geq f(1)=0$.

To prove the second inequality, we set the function on $0< x \leq 1$,
$$
g(x):=2x^{v+1}+v(1-v)(x-1)^2-2(1-v)x-2vx^2.
$$
By simple calculations, we have
$g'(x)=2(v+1)x^v+2v(1-v)(x-1)-2(1-v)-4vx$ and $g''(x)=2v(v+1)(x^{v-1}-1)\geq 0$ for $0< x \leq 1$ and $0\leq v \leq 1$.
Thus we have $g'(x) \leq g'(1) =0$ which implies $g(x) \geq g(1) =0$.
\end{proof}

Similar to the proof of Theorem \ref{theorem1.2}, we have the following.
\begin{theorem}\label{theorem1.3}
Let $A,B\in\mathcal{M}_n^+$ with $B\leq A$ and let $0\leq v\leq 1.$ Then
\begin{equation}\label{8}
\frac{v\left( 1-v \right)}{2}\left( A-B \right)A^{-1}(A!B)A^{-1}\left( A-B \right)\leq 
A{{\nabla }_{v}}B-A{{\sharp}_{v}}B \leq
\frac{v(1-v)}{2}(A-B)B^{-1}(A-B).
\end{equation}
\end{theorem}
Combining Theorems \ref{theorem1.2} and \ref{theorem1.3} and noting symmetry of $v(1-v)$ about $v=\frac{1}{2}$, we obtain \eqref{hirz_intro} as a corollary.

\begin{corollary}
Let $A,B\in\mathcal{M}_n^+$ be such that  $B\leq A$. Then
\begin{equation}\label{sing_v}
\frac{v\left( 1-v \right)}{2}{{s }_{j}}\left( {{A}^{-\frac{1}{2}}}{{\left( A-B \right)}^{2}}{{A}^{-\frac{1}{2}}} \right)\leq {{s }_{j}}\left( A{{\nabla }_{v}}B-A{{\sharp}_{v}}B \right)\le \frac{v\left( 1-v \right)}{2}{{s }_{j}}\left( {{B}^{-\frac{1}{2}}}{{\left( A-B \right)}^{2}}{{B}^{-\frac{1}{2}}} \right)
\end{equation}
\end{corollary}

On the other hand, when $\frac{1}{2}\leq v\leq 1,$ we have the following estimates. It should be remarked that the next estimates are better than thoes given in Lemma \ref{lemma_v}. These will help better see how \cite[Corollary 1]{2} is refined when $\frac{1}{2}\leq v\leq 1.$
\begin{lemma}\label{lemma_v_2}
If (i) $x\geq 1$ and ${1}/{2}\;\le v\le 1$ or (ii) $0< x \le 1$ and $0\le v \le 1/2$, then
\[
\left( 1-v \right)+vx-{{x}^{v}} \ge  \frac{v\left( 1-v \right)}{2}{{\left( 1-x \right)}^{2}}{{\left( \frac{1+x}{2} \right)}^{-1}}.
\]
\end{lemma}
\begin{proof}
We firstly consider the case (i).
For the given parameters, define
$$f(x)=\left( 1-v \right)+vx-{{x}^{v}}- \frac{v\left( 1-v \right)}{2}{{\left( 1-x \right)}^{2}}{{\left( \frac{1+x}{2} \right)}^{-1}}.$$
Direct calculus computations imply
$$f''(x)=v(1-v)g(v),\;{\text{where}}\;g(v)=\left(x^{v-2}-\frac{8}{(1+x)^3}\right).$$
Since $$g'(v)=x^{v-2}\log x,$$ it follows that $g$ is an increasing function of $v$ when $x\geq 1$.
When $v\geq\frac{1}{2}, x\geq 1$, we have $g(v)\geq g\left(\frac{1}{2}\right)= x^{-3/2}-\frac{8}{(1+x)^3}\geq 0.$  Since $g(v)\geq 0$, it follows that $f''(x)\geq 0$, when $x\geq 1$ and $0\leq v\leq \frac{1}{2}.$ Consequently, $f'(x)\geq f'(1)=0$ and hence, $f(x)\geq f(1)= 0.$ This shows that $f(x)\geq 0$ for all $x\geq 1$, which completes the proof of the first inequality.

Next, we consider the case (ii). For this case, we have $f''(x) \geq 0$ since $g'(v) \leq 0$ for $0<x \leq 1$ so that 
$g(v) \geq g(1/2)=x^{-3/2}-\dfrac{8}{(x+1)^3} \geq 0$. Thus we have
$f'(x) > f'(0) =v(1+3(1-v)) \geq 0$ which implies $f(x) >f(0) =(1-v)^2 \ge 0$. 
\end{proof}

We remark that the following inequality does not hold in general.
$$
(1-v)+vx-x^v \le \frac{v(1-v)}{2}\frac{(x-1)^2}{\sqrt{x}}
$$
for neither (i) $x\geq 1$ and ${1}/{2}\;\le v\le 1$ nor (ii) $0< x \le 1$ and $0\le v \le 1/2$.
Now proceeding with functional calculus argument as before implies the following matrix inequality, which we use next to obtain a refinement of \cite[Corollary 1]{2}.
\begin{corollary} \label{cor_2.5}
Let $A,B\in\mathcal{M}_n^+$. If (i) $A \le B$ and $\frac{1}{2}\leq v\leq 1$ or (ii) $A \geq B$ and $0 \le v \le 1/2$, then
\[ A\nabla_v B-A\sharp_v B \geq \frac{v(1-v)}{2}(B-A)\left(A\nabla B\right)^{-1}(B-A).\]
\end{corollary}
When $v=1/2$ in Corollary \ref{cor_2.5}, we recover Theorem \ref{thm1} under the condition $A\le B$.
Consequently, we obtain the following refinement of \cite[Corollary 1]{2}, for $\frac{1}{2}\leq v\leq 1.$

\begin{corollary}
Let $A,B\in\mathcal{M}_n^+$ be such that (i) $A \le B$ and $\frac{1}{2}\leq v\leq 1$ or (ii) $A \geq B$ and $0 \le v \le 1/2$. Then, for $j=1,2,\cdots,n,$
\begin{align*}
s_j\left(A\nabla_v B-A\sharp_v B\right)&\geq \frac{v(1-v)}{2} s_j\left((A-B)\left(A\nabla B\right)^{-1}(A-B)\right)\\
&\geq \frac{v(1-v)}{2} s_j\left(A^{-\frac{1}{2}} (A-B)^2 A^{-\frac{1}{2}}\right).
\end{align*}

\end{corollary}

\subsection{Geometric-Harmonic mean inequalities}
\begin{lemma}\label{lemma_GH_1}
Let $x>0$. Then 
$$\sqrt{x}-\left(\frac{1+\frac{1}{x}}{2}\right)^{-1}\leq \frac{(1-x)^2}{8\sqrt{x}}.$$
\end{lemma}
\begin{proof}
Let $$f(x)=\sqrt{x}-\left(\frac{1+\frac{1}{x}}{2}\right)^{-1}- \frac{(1-x)^2}{8\sqrt{x}}.$$ Direct computations imply
$$f''(x)=\frac{1}{32}\left(\frac{128}{(1+x)^3}-\frac{(3+x)(1+3x)}{x^{5/2}}\right).$$ Further computations yield
\begin{align*}
&((3 + x) (1 + 3 x) (1 + x)^3)^2 - (128)^2 x^5\\
&=(-1 + x)^2 (9 + 132 x + 868 x^2 + 3452 x^3 + 9510 x^4 + 3452 x^5 + 
   868 x^6 + 132 x^7 + 9 x^8)\\
   &\geq 0, x>0.
\end{align*}
Rearranging this last inequality implies
$$((3 + x) (1 + 3 x) (1 + x)^3)^2 - (128)^2 x^5\leq 0\Leftrightarrow \frac{128}{(1+x)^3}-\frac{(3+x)(1+3x)}{x^{5/2}}\leq 0.$$ This implies that $f''(x)\leq 0$ for all $x>0$. Consequently, $f'$ is decreasing on $(0,\infty).$ So, if $0<x\leq 1$ then $f'(x)\geq f'(1)=0$ and $f'(x)\leq f'(1)=0$ when $x\geq 1.$ This means that $f$ has a global maximum at $x=1$. That is, for all $x>0$, we must have $f(x)\leq f(1)=0,$ which completes the proof.
\end{proof}
Applying a functional calculus argument with $x=A^{-\frac{1}{2}}BA^{-\frac{1}{2}}$ in Lemma \ref{lemma_GH_1} implies the following bound for the difference between the geometric and harmonic matrix means.
\begin{theorem}
Let $A,B\in\mathcal{M}_n^+$. Then
$$A\sharp B-A!B\leq \frac{1}{8}(A-B)(A\sharp B)^{-1}(A-B).$$
\end{theorem}
Arguing as in the previous section, we reach the following singular values inequality for the difference $A\sharp B-A!B.$
\begin{corollary}\label{cor_GH_1}
Let $A,B\in \mathcal{M}_n^{+}$ be such that $B\leq A.$ Then
\begin{align*}
s_j\left(A\sharp B-A!B\right)&\leq \frac{1}{8}s_j\left((A-B)(A\sharp B)^{-1}(A-B)\right)\\
&\leq \frac{1}{8} s_j\left(B^{-\frac{1}{2}}(A-B)^2B^{-\frac{1}{2}}\right),
\end{align*}
for $j=1,2,\cdots,n.$
\end{corollary}
It is interesting that we have the same upper bound in Corollaries \ref{cor_AG} and \ref{cor_GH_1}.

Following the same theme of the previous section and Lemma \ref{lemma_GH_1}, we have the following generalization of Lemma \ref{lemma_GH_1}.
\begin{lemma}
Let $x\geq 1.$ Then
$$x^v-(1-v+vx^{-1})^{-1}\leq \frac{v(1-v)}{2}(1-x)^2x^{-v},$$ for $0\leq v\leq 1.$
\end{lemma}
This implies the matrix version:
\begin{theorem}
Let $A,B\in\mathcal{M}_n^+$ be such that $A\leq B.$ Then
$$A\sharp_v B-A!_v B\leq \frac{v(1-v)}{2}(A-B)(A\sharp_v B)^{-1}(A-B),$$ for $0\leq v\leq 1.$
\end{theorem}
This implies the following singular values inequality.
\begin{corollary}
Let $A,B\in\mathcal{M}_n^+$ be such that $A\leq B.$ Then, for $j=1,2,\cdots,n,$
\begin{align*}
s_j\left(A\sharp_v B-A!_v B\right)&\leq \frac{v(1-v)}{2}s_j\left((A-B)(A\sharp_v B)^{-1}(A-B)\right)\\
&\leq \frac{v(1-v)}{2}s_j\left(A^{-\frac{1}{2}}(A-B)^2A^{-\frac{1}{2}}\right).
\end{align*}

\end{corollary}

\subsection{Arithmetic-Harmonic mean inequalities}
We conclude this article by stating related results for the arithmetic-harmonic mean inequalities. The proofs are very similar to the above results, we we omit them.\\
Noting the identity 
$$\frac{1+x}{2}-\left(\frac{1+x^{-1}}{2}\right)^{-1}=\frac{(1-x)^2}{4}\left(\frac{1+x}{2}\right)^{-1}, x>0,$$ we obtain the following matrix versions.
\begin{theorem}
Let $A,B\in\mathcal{M}_n^+$. Then
$$A\nabla B-A!B=\frac{1}{4}(A-B)(A\nabla B)^{-1}(A-B).$$
In particular, if $A\leq B$, then
$$\frac{1}{4}(A-B)B^{-1}(A-B)\leq A\nabla B-A!B \leq \frac{1}{4}(A-B)A^{-1}(A-B).$$
Consequently, for $j=1,2,\cdots,n,$ the following holds
$$\frac{1}{4}s_j\left(B^{-\frac{1}{2}}(A-B)^2B^{-\frac{1}{2}}\right)\leq s_j\left(A\nabla B-A!B\right)\leq \frac{1}{4}s_j\left(A^{-\frac{1}{2}}(A-B)^2A^{-\frac{1}{2}}\right),$$ when $A,B\in\mathcal{M}_n^+$ are such that $A\leq B.$
\end{theorem}

\vskip 0.3 true cm

{\tiny (M. Sababheh) Department of basic sciences, Princess Sumaya University for Technology, Amman 11941, Jordan.}

{\tiny \textit{E-mail address:} sababheh@psut.edu.jo}

{\tiny \vskip 0.3 true cm }

{\tiny (S. Furuichi) 
Department of Information Science, College of Humanities and Sciences, Nihon University, 3-25-40, Sakurajyousui, Setagaya-ku,
Tokyo, 156-8550, Japan}

{\tiny \textit{E-mail address:} furuichi@chs.nihon-u.ac.jp}

{\tiny \vskip 0.3 true cm }

{\tiny (S. Sheybani) Department of Mathematics, Mashhad Branch, Islamic Azad University, Mashhad, Iran}

{\tiny \textit{E-mail address:} shiva.sheybani95@gmail.com}

{\tiny \vskip 0.3 true cm }

{\tiny (H. R. Moradi) Department of Mathematics, Payame Noor University (PNU), P.O. Box 19395-4697, Tehran, Iran}

{\tiny \textit{E-mail address:} hrmoradi@mshdiau.ac.ir }
\end{document}